%% file: main.tex
\documentclass[10pt]{amsart}
\usepackage[utf8]{inputenc}
\usepackage{times}
\input{header}

\title[Lifting to $\BP{n-1}$ and H-dR degeneration]{Lifting to truncated Brown-Peterson spectra and Hodge-de Rham degeneration in characteristic $p>0$}
\author{S. K. Devalapurkar}
\address{1 Oxford St, Cambridge, MA 02139}
\email{sdevalapurkar@math.harvard.edu, \today}
\thanks{Part of this work was done when the author was supported by the PD Soros Fellowship and NSF DGE-2140743. I'm grateful to Ben Antieau, Bhargav Bhatt, Jeremy Hahn, and the referee for suggestions which improved this note.}

\begin{document}

\maketitle

\begin{abstract}
    The goal of this note is to prove that Hodge-de Rham degeneration holds for smooth and proper $\FF_p$-schemes $X$ with $\dim(X)<p^n$ assuming that two conditions hold: its category of quasicoherent sheaves admits a lift to the truncated Brown-Peterson spectrum $\BP{n-1}$; and the Hochschild-Kostant-Rosenberg spectral sequence for $X$ degenerates at the $E_2$-page. This result is obtained from a noncommutative version thereof, whose proof is essentially the same as Mathew's argument in \cite{mathew-kaledin}.
\end{abstract}

Let $X$ be a smooth and proper scheme over a perfect field $k$ of characteristic $p>0$. In \cite{deligne-illusie}, Deligne and Illusie proved that the Hodge decomposition holds for the de Rham cohomology of $X$ under certain hypotheses: namely, if $\dim(X)<p$ and $X$ admits a smooth and proper lift to the truncated Witt vectors $W_2(k) = W(k)/p^2$, they showed that the Hodge-de Rham spectral sequence
$$E_1^{\ast,\ast} = \H^\ast(X; \Omega^\ast_{X/k}) \Rightarrow \H^\ast_\dR(X/k)$$
collapses at the $E_1$-page.

In \cite[Remarque 2.6(iii)]{deligne-illusie} (see also \cite[Problem 7.10]{illusie-frob-and-hodge}), Deligne and Illusie asked if the Hodge-de Rham spectral sequence could degenerate for a smooth proper $k$-scheme $X$ with a lift to $W(k)/p^2$ (or even to $W(k)$), without any dimension assumptions. This remarkable question has recently been answered (in the negative) by Sasha Petrov in \cite{petrov-hdr}. Our goal in this note is to study conditions on $X$ arising from chromatic homotopy theory which \textit{do} guarantee Hodge-de Rham degeneration if $\dim(X)>p$.

\begin{recall}
Let $X$ be a smooth scheme over a commutative ring $k$. One then has the HKR and de-Rham-to-$\HP$ spectral sequences (see \cite[Definition 3.1]{antieau-bhatt-mathew-counterexample}):
\begin{align*}
    E_2^{s,t} = \H^s(X; \wedge^{-t} L_{X/k}) & \Rightarrow \pi_{-(s+t)} \HH(X/k), \\
    E_2^{s,t} = \H^{s-t}_\dR(X/k) & \Rightarrow \pi_{-(s+t)} \HP(X/k).
\end{align*}
There are also the Hodge-de Rham and the Tate spectral sequences
\begin{align*}
    E_1^{s,t} = \H^s(X; \wedge^{t} L_{X/k}) & \Rightarrow \H^{s+t}_\dR(X/k), \\
    E_2^{s,t} = \hat{\H}^s(BS^1; \pi_t \HH(\cC/\FF_p)) & \Rightarrow \pi_{t-s} \HP(\cC/\FF_p),
\end{align*}
where $\hat{\H}$ denotes Tate cohomology. Note that if we write $\H^\ast(BS^1; \FF_p) = \FF_p[\hbar]$ with $\hbar$ in cohomological degree $2$, then the $E_2$-page of the Tate spectral sequence can be rewritten as $\pi_\ast \HH(\cC/\FF_p)[\hbar^{\pm 1}]$.
\end{recall}
Let $n \leq \infty$. Fix an $\E{3}$-form of the ($p$-completed) truncated Brown-Peterson spectrum $\BP{n-1}$ of height $n-1$, which exists thanks to \cite[Theorem A]{hahn-wilson-bpn}. By construction, $\pi_\ast \BP{n-1} \cong \Z_p[v_1, \cdots, v_{n-1}]$ for classes $v_i$ in degree $2p^i-2$. By convention, $\BP{-1} = \FF_p$. We also have $\BP{0} = \Z_p$, and $\BP{1}$ can be identified with the connective cover of the Adams summand of $p$-completed complex K-theory. There is also a tight relationship between $\BP{2}$ and elliptic cohomology. When $n = \infty$, the $\E{3}$-ring $\BP{\infty}$ is denoted $\BPP$, and is called the Brown-Peterson spectrum.

Our goal in this note is to prove:
\begin{theorem}\label{hdr-degen}
Let $n \leq \infty$, and let $X$ be a smooth and proper scheme over\footnote{Here, $\FF_p$ could be replaced by any perfect field of characteristic $p>0$; we only use $\FF_p$ to avoid introducing conceptually unnecessary notation.} $\FF_p$ of dimension $<p^n$. Suppose that:
\begin{enumerate}
    \item The HKR spectral sequence degenerates at the $E_2$-page; and
    \item $\QCoh(X)$ lifts to a smooth and proper left $\BP{n-1}$-linear $\infty$-category\footnote{Recall that at the beginning of this article, we picked an $\E{3}$-form of $\BP{n-1}$, which exists by \cite[Theorem A]{hahn-wilson-bpn}. Then, a ``left $\BP{n-1}$-linear $\infty$-category'' is simply a left $\LMod_{\BP{n-1}}$-module in $\mathrm{Pr^L}$, where $\LMod_{\BP{n-1}}$ is equipped with the $\E{2}$-monoidal structure arising from the $\E{3}$-structure on $\BP{n-1}$. See \cite[Variant D.1.5.1]{SAG}.}.
\end{enumerate}
Then the Hodge-de Rham spectral sequence
$$E_1^{\ast,\ast} = \H^\ast(X; \Omega^\ast_{X/\FF_p}) \Rightarrow \H^\ast_\dR(X/\FF_p)$$
collapses at the $E_1$-page, and the de-Rham-to-$\HP$ spectral sequence collapses at the $E_2$-page.
\end{theorem}

The discussion in \cite[Remark 3.6]{antieau-bhatt-mathew-counterexample} implies that if the HKR and Tate spectral sequences both degenerate, then both the Hodge-de Rham and de Rham-to-$\HP$ spectral sequences must also degenerate. It therefore suffices to prove the following noncommutative statement\footnote{Our original proof used the higher chromatic topological Sen operators from our forthcoming article \cite{thh-xn} to argue in a manner similar to \cite[Example 4.7.17]{apc}, but we soon realized that the argument could be simplified much further. In \cite[Remark C.14]{thh-xn}, we also phrase an analogue of \cref{higher-dim-degen} in stacky language via the Sen operator of \cite{apc} and the stack $BW^\times[F^n]$. The expected isomorphism, which we hope to study in joint work with Jeremy Hahn and Arpon Raksit, between $BW^\times[F^n]$ and the stack associated to the motivic filtration of $\THH(\BP{n-1})^{t\Cp}/(p, \cdots,v_{n-1})$ was the original motivation for our result.}:
\begin{prop}\label{higher-dim-degen}
Let $n\leq \infty$, and let $\cC$ be a smooth and proper $\FF_p$-linear $\infty$-category such that $\pi_j \HH(\cC/\FF_p) = 0$ for $j\not\in [-p^n,p^n]$. If $\cC$ lifts to a smooth and proper left $\BP{n-1}$-linear $\infty$-category, then the Tate spectral sequence
$$E_2^{\ast,\ast} = \hat{\H}^\ast(BS^1; \pi_\ast \HH(\cC/\FF_p)) \Rightarrow \pi_\ast \HP(\cC/\FF_p)$$
collapses at the $E_2$-page.
\end{prop}

\begin{remark}
When $n=1$, \cref{hdr-degen} is part of the main result of \cite{deligne-illusie}\footnote{As the reader may have noticed, the title of this work is a tribute to the inspirational paper \cite{deligne-illusie}.}: in this case, condition (b) in \cref{hdr-degen} is asking for a lifting to $\BP{0} = \Z_p$. As mentioned above, Sasha Petrov recently constructed in \cite{petrov-hdr} a $(p+1)$-dimensional smooth and proper $\Z_p$-scheme $\fr{X}$ such that the Hodge-de Rham spectral sequence for its special fiber $\fr{X}_{p=0}$ does not degenerate at the $E_1$-page. If the HKR spectral sequence degenerates at the $E_2$-page for Petrov's $\fr{X}_{p=0}$, then $\QCoh(\fr{X})$ provides an example of a $\Z_p$-linear $\infty$-category which cannot lift to a $\ku$-linear $\infty$-category.

We view \cref{hdr-degen} as a step towards a positive answer of Deligne and Illusie's question in some generality. Note that condition (b) in \cref{hdr-degen} is significantly weaker than asking that $X$ itself admit some sort of lifting as a spectral scheme. Note, also, that we do not prove anything nearly as refined as \cite{deligne-illusie}: namely, we do not provide any sort of correspondence between liftings and splittings of truncations of the de Rham complex. For instance, it would be very interesting if, for a $\Z_p$-scheme $\fr{X}$, there were a relationship between splittings of the mod $p$ reduction $\diffr_{\fr{X},0} \otimes_{\Z_p} \FF_p$ of the zeroth generalized eigenspace of the diffracted Hodge complex (see \cite[Remark 4.7.20]{apc} for this notion) and liftings of $\QCoh(\fr{X})$ to $\BP{1}$.
\end{remark}

\begin{remark}
Let $I = (p^2, v_1^2, \cdots, v_{n-1}^2)$. Were $\BP{n-1}/I$ to admit the structure of an $\E{2}$-ring, \cref{hdr-degen} (and \cref{higher-dim-degen}) would continue to hold with $\BP{n-1}$ replaced by $\BP{n-1}/I$. This is because one can prove that \cref{injective-sigma} continues to hold for $\BP{n-1}/I$.

Some preliminary calculations seem to suggest that Petrov's first Sen class (see \cite{petrov-hdr, illusie-petrov-slides}) is related to the obstruction in Hochschild cohomology to lifting a $\Z_p$-scheme $\fr{X}$ along the map $\BP{1}/v_1^2 \to \Z_p$ (and even along the map $\tau_{\leq 2p-3} j \to \Z_p$, where $j$ is the connective complex image-of-$J$ spectrum). For instance, the first $k$-invariant of $\BP{1}/v_1^2$ is given by the map $\Z_p \to \Z_p[2p-1]$ defined via the composite
$$\Z_p \to \FF_p \xar{P^1} \FF_p[2p-2] \xar{\beta} \Z_p[2p-1],$$
where $P^1$ is a Steenrod operation and $\beta$ is the Bockstein. In other words, $\BP{1}/v_1^2$ is equivalent to the fiber of the above composite.
On the other hand, the extension class for $\co_\fr{X} \to \F^p \diffr_{\fr{X},0} \to L\Omega^p_\fr{X}[-p]$ is computed in \cite[Lemma 6.5]{petrov-hdr} to be the composite
$$L\Omega^p_\fr{X}[-p] \to L\Omega^p_{\fr{X}_{p=0}/\FF_p} [-p] \xar{c_{X,p}} \co_{\fr{X}_{p=0}} \xar{\beta} \co_\fr{X}[1],$$
where the ``first Sen class'' $c_{X,p}$ can be defined using Steenrod operations on cosimplicial algebras via \cite[Theorem 7.1]{petrov-hdr}.
We hope to explore this further to obtain a tighter connection between the results in this article and those of Petrov's.
\end{remark}
\begin{remark}
\cref{hdr-degen} has the following counter-intuitive consequence: if the HKR spectral sequence for $X$ degenerates at the $E_2$-page, then the differentials in the Hodge-de Rham spectral sequence obstruct the lifting of $\QCoh(X)$ to a smooth and proper left $\BP{n-1}$-linear $\infty$-category.
In particular, taking $n = \infty$, the condition in \cref{higher-dim-degen} that $\pi_j \HH(\cC/\FF_p) = 0$ for $j\not\in [-p^n,p^n]$ is vacuous; so we find that if $\cC$ is a smooth and proper $\FF_p$-linear $\infty$-category which admits a smooth and proper lift to $\BPP$, then its Tate spectral sequence collapses at the $E_2$-page. 

This was already known if $\cC$ lifts all the way to $S^0$; see \cite[Example 3.5]{mathew-kaledin}. In particular, therefore, one class of $X$ for which $\QCoh(X)$ does satisfy the hypotheses of \cref{higher-dim-degen} and \cref{hdr-degen} are toric varieties; but in those cases, degeneration was already known for $X$ of arbitrary dimension (since they are $F$-liftable).
Interesting examples of \cref{hdr-degen} and \cref{higher-dim-degen} are currently lacking, but one would be most welcome.
\end{remark}
\begin{remark}
One could also ask the following question: if $n\geq 0$, is there an example of a smooth and proper $\BP{n-1}$-linear $\infty$-category $\cC$ which does not lift to a smooth and proper left $\BP{n}$-linear $\infty$-category?
\end{remark}
The idea to prove \cref{higher-dim-degen} is essentially the argument of \cite{mathew-kaledin}, so we recommend reading that paper first. Recall B\"okstedt's calculation that $\pi_\ast \THH(\FF_p) \cong \FF_p[\sigma]$, where $\sigma$ lives in degree $2$. By \cite[Proposition 3.4]{mathew-kaledin}, \cref{higher-dim-degen} is a consequence of:
\begin{prop}\label{sigma-tf}
Let $\cC$ be a smooth and proper $\FF_p$-linear $\infty$-category such that $\pi_j \HH(\cC/\FF_p) = 0$ for $j\not\in [-p^n,p^n]$. If $\cC$ lifts to a smooth and proper left $\BP{n-1}$-linear $\infty$-category, then $\THH(\cC)$ is $\sigma$-torsionfree.
\end{prop}
To prove \cref{sigma-tf}, we need a preliminary result. It follows from \cite[Theorem 5.2 and Corollary 2.8]{framed-e2} that there is an augmentation $\THH(\BP{n-1}) \to \BP{n-1}$; composing with the map $\BP{n-1} \to \FF_p$ defines a map $\THH(\BP{n-1}) \to \FF_p$.
\begin{prop}\label{prop: map factors through Fp}
    The map $\tau_{\leq 2p^n-1} \THH(\BP{n-1}) \to \tau_{\leq 2p^n-1}\THH(\FF_p)$ factors, as an $\E{2}$-algebra map, as the composite 
    $$\tau_{\leq 2p^n - 1} \THH(\BP{n-1}) \to \FF_p \to \tau_{\leq 2p^n-1} \THH(\FF_p).$$
\end{prop}
\begin{proof}
    It evidently suffices to show that the map 
    $$\tau_{\leq 2p^n-1} (\THH(\BP{n-1}) \otimes_{\BP{n-1}} \FF_p) \to \tau_{\leq 2p^n-1}\THH(\FF_p)$$
    factors, as an $\E{2}$-algebra map, as the composite 
    $$\tau_{\leq 2p^n - 1} (\THH(\BP{n-1}) \otimes_{\BP{n-1}} \FF_p) \to \FF_p \to \tau_{\leq 2p^n-1} \THH(\FF_p).$$
    There is an $\E{3}$-map $\BPP \to \BP{n}$, which defines an $\E{2}$-map 
    $$\THH(\BPP) \otimes_\BPP \FF_p \to \THH(\BP{n-1}) \otimes_{\BP{n-1}} \FF_p.$$
    This map is an equivalence in degrees $\leq 2p^n - 1$.\footnote{For instance, this follows from \cite[Proposition 2.9]{thh-truncated-bp} (see also \cite[Remark 2.2.5]{thh-xn}), which says that for $n\leq \infty$, there is an isomorphism
    $$\pi_\ast (\THH(\BP{n-1}) \otimes_{\BP{n-1}} \FF_p) \cong \FF_p[\sigma^2(v_n)] \otimes \Lambda(\sigma(t_1), \cdots, \sigma(t_n)),$$
    where $|\sigma^2(v_n)| = 2p^n$ and $|\sigma(t_i)| = 2p^i-1$.} Therefore, it suffices to show that the map $\THH(\BPP) \otimes_{\BPP} \FF_p \to \THH(\FF_p)$ factors, as an $\E{2}$-map, as the composite 
    $$\THH(\BPP) \otimes_{\BPP} \FF_p \to \FF_p \to \THH(\FF_p);$$
    equivalently, that the map $\THH(\BPP) \to \THH(\FF_p)$ factors, as an $\E{2}$-map, as the composite
    $$\THH(\BPP) \to \BPP \to \THH(\FF_p).$$
    Here, the map $\BPP \to \THH(\FF_p)$ is just the composite of the map $\BPP \to \FF_p$ with the unit $\FF_p \to \THH(\FF_p)$.
    Since $\BPP$ is an $\E{4}$-algebra retract of $\MU$ (compatibly with their natural maps to $\FF_p$), it suffices to replace $\BPP$ by $\MU$ in the above discussion; in fact, we will even show that the map $\THH(\MU) \to \THH(\FF_p)$ factors, as an $\E{3}$-map, as the composite
    $$\THH(\MU) \to \MU \to \THH(\FF_p).$$
    Here, the map $\MU \to \THH(\FF_p)$ is just the composite of the map $\MU \to \FF_p$ with the unit $\FF_p \to \THH(\FF_p)$.

    Recall from \cite{thh-thom} and \cite{klang} that there is an equivalence $\THH(\MU) \simeq \MU[\SU]$ of $\Eoo$-$\MU$-algebras, and that the augmentation $\THH(\MU) \to \MU$ is given by taking $\MU$-chains of the augmentation $\SU \to \ast$. The $\Eoo$-$\MU$-linear map $\THH(\MU) \to \THH(\FF_p)$ is therefore equivalent to the data of an $\Eoo$-map $\SU \to \GL_1(\THH(\FF_p))$. Since $\THH(\FF_p)$ is concentrated in even degrees, $\GL_1(\THH(\FF_p))$ is an $\Eoo$-space with even homotopy. It therefore suffices to prove the following claim: any $\E{3}$-map $f: \SU \to X$ to an $\E{3}$-space $X$ with even homotopy factors (as an $\E{3}$-map) through the augmentation $\SU \to \ast$. Indeed, $f$ is equivalent to the data of a map $\B^3 f: \B^3 \SU \to \B^3 X$. Since $\B^3 \SU = \BU\pdb{6}$ has an even cell decomposition and $\B^3 X$ has odd homotopy, the map $\B^3 f$ is necessarily null (so $f$ is null as an $\E{3}$-map), as desired.
\end{proof}
The proof of the following result is a direct adaptation of that of \cite[Proposition 3.7]{mathew-kaledin}; it could also be proved using the methods of \cite{thh-xn}.
\begin{lemma}\label{injective-sigma}
Let $M$ be a perfect $\THH(\FF_p)$-module such that $\pi_i(M) = 0$ for $i<a$. If $M$ lifts to a perfect $\THH(\BP{n-1})$-module $\tilde{M}$, then $\sigma$-multiplication $\sigma: \pi_{i-2} M \to \pi_i M$ is injective for 
$i\leq a + 2p^n-1$.
\end{lemma}
\begin{proof}
To prove the result of the lemma, we can assume without loss of generality that $a=0$. Then, there is a map
$$M \to \tau_{\leq 2p^n-1} \tilde{M} \otimes_{\tau_{\leq 2p^n-1} \THH(\BP{n-1})} \tau_{\leq 2p^n-1} \THH(\FF_p),$$
which is an equivalence on $\tau_{\leq 2p^n-1}$. By \cref{prop: map factors through Fp}, the map $\tau_{\leq 2p^n-1} \THH(\BP{n-1}) \to \tau_{\leq 2p^n-1}\THH(\FF_p)$ factors through $\FF_p \to \tau_{\leq 2p^n-1} \THH(\FF_p)$, so we see that $\tau_{\leq 2p^n-1} M$ is a free $\tau_{\leq 2p^n-1} \THH(\FF_p)$-module on classes in nonnegative degrees. Therefore, $\sigma$-multiplication is injective through the stated range.
\end{proof}
\cref{sigma-tf} is now a consequence of the following, whose proof is a direct adaptation of that of \cite[Proposition 3.8]{mathew-kaledin}.
\begin{prop}\label{freeness}
Let $M$ be a perfect $\THH(\FF_p)$-module with Tor-amplitude in $[-p^n,p^n]$. If $M$ lifts to a perfect $\THH(\BP{n-1})$-module $\tilde{M}$, then $M$ is free.
\end{prop}
\begin{proof}
The argument is the same as in \cite[Proposition 3.8]{mathew-kaledin}. Indeed, $M$ is a direct sum of $\THH(\FF_p)$-modules which are free or of the form $M_{i,j} = \Sigma^i \THH(\FF_p)/\sigma^j$ (see \cite[Proposition 3.3]{mathew-kaledin}). Since $M_{i,j}$ has Tor-amplitude in $[i,i+2j+1]$, the condition on $M$ implies that $M_{i,j}$ could appear as a summand of $M$ if and only if $-p^n\leq i \leq i+2j+1 \leq p^n$.

The class $\sigma^{j-1}[i]\in \pi_{i+2j-2} M_{i,j}$ is killed by $\sigma$, so taking $a = -p^n$ in \cref{injective-sigma}, we see that
$$i+2j > -p^n + 2p^n - 1 = p^n-1.$$
In particular, $i+2j+1 > p^n$, which contradicts $i+2j+1\leq p^n$. Therefore, no $M_{i,j}$ can be a summand of $M$, so that $M$ is free.
\end{proof}

In the remainder of this note, we will clarify the relationship between liftings of $X$ itself and Hodge-de Rham degeneration. First, observe that assumption (b) in \cref{hdr-degen} is only a condition on $\QCoh(X)$, which is essentially why \cref{higher-dim-degen} is the more natural noncommutative statement. It seems to me that assumption (a) in \cref{hdr-degen} could be removed if we asked that the structure sheaf $\co_X$ itself lifted to a sheaf of $\E{2}$-$\BP{n-1}$-algebras.

One could ask about lifting $X$ itself as an $\Eoo$-spectral scheme in the current setup \cite{SAG} of spectral algebraic geometry.
Unfortunately, this question often does not make sense, since $\BP{n-1}$ is generally not an $\Eoo$-ring \cite{lawson-bp, senger-bp}. 
Nevertheless, the question does make sense if, for instance, $n=2$ (since $\BP{1}$ is an $\Eoo$-ring). In this case, requiring that $X$ lift is significantly stronger than the assumptions of \cref{hdr-degen}, as shown by the following.
\begin{prop}\label{ku-lifting}
Let $X$ be a smooth and proper $\FF_p$-scheme. If $X$ lifts to a $p$-adic flat $\ku^\wedge_p$-scheme $\fr{X}$, then the Hodge-de Rham spectral sequence for $X$ degenerates at the $E_1$-page.
\end{prop}
\begin{proof}
The lift $\fr{X}$ defines a lift of $X$ to $\Z_p$ via $\fr{X}_0 := \fr{X} \otimes_{\ku^\wedge_p} \Z_p$. It suffices to show that $\fr{X}_0$ admits a $\delta$-ring structure; then, the Hodge-Tate gerbe over $\fr{X}_0$ (from \cite[Proposition 5.12]{prismatization}) splits, so that the conjugate (and hence Hodge-de Rham) spectral sequence for $X$ degenerates. The fact that $\fr{X}$ is assumed to be flat implies that $\pi_0 \Lone \co_{\fr{X}} \cong \pi_0 \co_{\fr{X}} = \co_{\fr{X}_0}$. By \cite{k1local}, if $R$ is any $K(1)$-local $\Eoo$-ring, then $\pi_0(R)$ admits a $\delta$-ring structure (functorially in $R$). Globalizing, we see that $\pi_0 \Lone \co_{\fr{X}} = \co_{\fr{X}_0}$ has a $\delta$-ring structure, which implies the desired claim.
\end{proof}
\begin{remark}
It follows from \cref{ku-lifting} that lifting an arbitrary-dimensional $X$ to a $\ku^\wedge_p$-scheme suffices to conclude Hodge-de Rham degeneration; in particular, this assumption is significantly stronger than those of \cref{hdr-degen}. One intermediate between the assumptions of \cref{ku-lifting} and \cref{hdr-degen} is the following: one could assume that $\co_X$ only admit a lift to a sheaf of $\E{m}$-$\BP{n-1}$-algebras (whenever this makes sense). \cref{ku-lifting} corresponds to the case $n=2$ and $m=\infty$, while \cref{hdr-degen} roughly corresponds to the case $m=1$ (and $n$ arbitrary). What constraints does such a lifting impose on the Hodge-de Rham spectral sequence for $X$? For instance, if $p$ is an odd prime, and $\co_X$ admits a flat lift to a sheaf of $\E{2n+1}$-$\ku^\wedge_p$-algebras, then the general construction of power operations (following \cite{k1local}) along with the equivalence $\Lone \Conf^\un_p(\RR^{2n+1}) \simeq \Lone S^{-1}/p^n$ of \cite{davis-odd-primary-bo} shows that $\fr{X}_0$ has a lift of Frobenius modulo $p^{n+1}$.
%
In particular, if $\co_X$ admits a flat lift to a sheaf of $\E{3}$-$\ku^\wedge_p$-algebras, and $\dim(X)<p$, then \cite{deligne-illusie} implies that the Hodge-de Rham spectral sequence degenerates for $X$.
\end{remark}
\begin{remark}
Finally, one might wonder whether a lifting of $X$ to $\BP{n-1}$, or $\ku^\wedge_p$, or even the sphere spectrum can be used to prove that the HKR spectral sequence degenerates. Unfortunately, it seems that there is no clear relationship between HKR degeneration and liftings to the sphere. For instance, the stack $B\mu_p$ over $\Z_p$ lifts to the $p$-complete sphere spectrum (by writing $\mu_p = \spec S[\Z/p]$), but the HKR spectral sequence for $B\mu_p$ does not degenerate by \cite[Theorem 4.6]{antieau-bhatt-mathew-counterexample}. Nevertheless, there are some liftability and torsion-freeness criteria, such as those described by Antieau-Vezzosi in \cite[Remark 1.6 and Example 1.7]{antieau-vezzosi}, which do guarantee HKR degeneration.
\end{remark}

\bibliographystyle{alpha}
\bibliography{main}
\end{document}

%% file: header.tex
\usepackage[dvipsnames]{xcolor}
\usepackage{todonotes}

\usepackage{etoolbox}

\makeatletter
\patchcmd{\@thm}{\let\thm@indent\indent}{\let\thm@indent\noindent}{}{}
\patchcmd{\@thm}{\thm@headfont{\scshape}}{\thm@headfont{\bfseries}}{}{}

\usepackage[T1]{fontenc}

\usepackage{hyperref}

\usepackage{mathtools}

\usepackage{multirow}
\usepackage{longtable}

\usepackage{amsmath}
\usepackage{amsfonts}
\usepackage{amssymb}
\usepackage{float}
\usepackage{amsthm}
\usepackage{comment}
\usepackage{amscd}
\usepackage{amsxtra}
\usepackage{epsfig}
\usepackage{epigraph}
\usepackage{pgfplots}

\usepackage{stmaryrd}

\usepackage{listings}
\usepackage{color}

\definecolor{dkgreen}{rgb}{0,0.6,0}
\definecolor{gray}{rgb}{0.5,0.5,0.5}
\definecolor{mauve}{rgb}{0.58,0,0.82}

\usepackage{color}
\usepackage[all]{xy}
\usepackage{verbatim}

\usepackage{eucal}
\usepackage{mathrsfs}

\usepackage{graphicx}
\usepackage{tikz-cd}

\usepackage{url}
\usepackage{verbatim}

\usepackage{xy}
\xyoption{all}
\newcommand{\xar}[1]{\xrightarrow{{#1}}}

\usepackage{amsthm}
\usepackage{hhline}

\usepackage{caption} 
\makeatletter
\def\@tocline#1#2#3#4#5#6#7{\relax
  \ifnum #1>\c@tocdepth 
  \else
    \par \addpenalty\@secpenalty\addvspace{#2}%
    \begingroup \hyphenpenalty\@M
    \@ifempty{#4}{%
      \@tempdima\csname r@tocindent\number#1\endcsname\relax
    }{%
      \@tempdima#4\relax
    }%
    \parindent\z@ \leftskip#3\relax \advance\leftskip\@tempdima\relax
    \rightskip\@pnumwidth plus4em \parfillskip-\@pnumwidth
    #5\leavevmode\hskip-\@tempdima
      \ifcase #1
       \or\or \hskip 1em \or \hskip 2em \else \hskip 3em \fi%
      #6\nobreak\relax
    \hfill\hbox to\@pnumwidth{\@tocpagenum{#7}}\par
    \nobreak
    \endgroup
  \fi}
\makeatother

\usepackage[noabbrev,capitalize]{cleveref} 

\crefformat{nul}{(#2#1#3)}
\Crefformat{nul}{(#2#1#3)}

\crefname{section}{\S}{\S\S}
\crefname{subsection}{\S}{\S\S}
\crefname{axioms}{Axiom}{Axioms}
\crefname{exercise}{Exercise}{Exercises}
\crefname{exercisenum}{Exercise}{Exercises}
\crefname{construction}{Construction}{Constructions}
\crefname{problem}{Problem}{Problems}
\crefname{theorem}{Theorem}{Theorems}
\crefname{definition}{Definition}{Definitions}
\crefname{prop}{Proposition}{Propositions}
\crefname{lemma}{Lemma}{Lemmas}
\crefname{example}{Example}{Examples}
\crefname{examplealph}{Example}{Examples}
\crefname{corollary}{Corollary}{Corollaries}
\crefname{nonexample}{Nonexample}{Nonexamples}
\crefname{equation}{}{}
\crefname{summary}{Summary}{Summaries}
\crefname{recollection}{Recollection}{Recollections}
\Crefname{recollection}{Recollection}{Recollections}
\Crefname{nonexample}{Nonexample}{Nonexamples}
\Crefname{corollary}{Corollary}{Corollaries}
\Crefname{corollary}{Corollary}{Corollaries}
\Crefname{axioms}{Axiom}{Axioms}
\Crefname{exercise}{Exercise}{Exercises}
\Crefname{exercisenum}{Exercise}{Exercises}
\Crefname{construction}{Construction}{Constructions}
\Crefname{problem}{Problem}{Problems}
\Crefname{theorem}{Theorem}{Theorems}
\Crefname{definition}{Definition}{Definitions}
\Crefname{prop}{Proposition}{Propositions}
\Crefname{lemma}{Lemma}{Lemmas}
\Crefname{example}{Example}{Examples}
\Crefname{examplealph}{Example}{Examples}
\Crefname{section}{\S}{\S\S}
\Crefname{subsection}{\S}{\S\S}

\DeclareMathOperator{\spec}{Spec}

\DeclareMathOperator{\QCoh}{\mathrm{QCoh}}

\newtheorem{lemma}{Lemma}

\newtheorem{theorem}[lemma]{Theorem}

\newtheorem{prop}[lemma]{Proposition}

\newtheorem*{conj-moore}{Conjecture~\ref{moore-splitting}}
\newtheorem*{conj-cent}{Conjecture~\ref{centrality-conj}}
\newtheorem*{conj-tmf}{Conjecture~\ref{tmf-conj}}
\newtheorem*{thm-main}{Theorem~\ref{main-thm}}
\newtheorem*{cor-bpn}{Corollary~\ref{bpn}}
\newtheorem*{thm-string}{Theorem~\ref{mstring}}

\theoremstyle{definition}

\newtheorem{remark}[lemma]{Remark}

\newtheorem{recall}[lemma]{Recollection}

\newcommand{\FF}{\mathbf{F}}
\newcommand{\Z}{\mathbf{Z}}

\newcommand{\cC}{\mathcal{C}}

\newcommand{\RR}{\mathbf{R}}

\newcommand{\co}{\mathcal{O}}

\newcommand{\LMod}{\mathrm{LMod}}

\newcommand{\Lone}{L_{K(1)}}

\newcommand{\Eoo}{{\mathbf{E}_\infty}}

\newcommand{\E}[1]{{\mathbf{E}_{{#1}}}}

\newcommand{\GL}{\mathrm{GL}}

\newcommand{\MU}{\mathrm{MU}}
\newcommand{\BU}{\mathrm{BU}}

\newcommand{\ku}{\mathrm{ku}}

\newcommand{\BP}[1]{\mathrm{BP}\langle{#1}\rangle}

\newcommand{\SU}{\mathrm{SU}}
\newcommand{\BPP}{\mathrm{BP}}

\renewcommand{\H}{\mathrm{H}}

\newcommand{\Conf}{\mathrm{Conf}}

\newcommand{\F}{\mathrm{F}}
\newcommand{\B}{\mathrm{B}}

\newcommand{\HH}{\mathrm{HH}}

\newcommand{\HP}{\mathrm{HP}}
\newcommand{\THH}{\mathrm{THH}}

\newcommand{\dR}{\mathrm{dR}}

\newcommand{\fr}[1]{\mathfrak{#1}}

\newcommand{\un}{\mathrm{un}}

\renewcommand{\S}{Section }

\makeatletter
\providecommand{\leftsquigarrow}{%
  \mathrel{\mathpalette\reflect@squig\relax}%
}
\newcommand{\reflect@squig}[2]{%
  \reflectbox{$\m@th#1\rightsquigarrow$}%
}
\makeatother

\newcommand{\Cp}{{\Z/p}}

\renewcommand{\tilde}{\widetilde}

\usepackage{slashed}
\newcommand{\diffr}{\widehat{\Omega}^{\slashed{D}}}

\usepackage{relsize}
\usepackage[bbgreekl]{mathbbol}
\usepackage{amsfonts}
\DeclareSymbolFontAlphabet{\mathbb}{AMSb} 
\DeclareSymbolFontAlphabet{\mathbbl}{bbold}

\newcommand{\pdb}[1]{\langle{#1}\rangle}

\DeclareMathSymbol{A}{\mathalpha}{operators}{`A}
\DeclareMathSymbol{B}{\mathalpha}{operators}{`B}
\DeclareMathSymbol{C}{\mathalpha}{operators}{`C}
\DeclareMathSymbol{D}{\mathalpha}{operators}{`D}
\DeclareMathSymbol{E}{\mathalpha}{operators}{`E}
\DeclareMathSymbol{F}{\mathalpha}{operators}{`F}
\DeclareMathSymbol{G}{\mathalpha}{operators}{`G}
\DeclareMathSymbol{H}{\mathalpha}{operators}{`H}
\DeclareMathSymbol{I}{\mathalpha}{operators}{`I}
\DeclareMathSymbol{J}{\mathalpha}{operators}{`J}
\DeclareMathSymbol{K}{\mathalpha}{operators}{`K}
\DeclareMathSymbol{L}{\mathalpha}{operators}{`L}
\DeclareMathSymbol{M}{\mathalpha}{operators}{`M}
\DeclareMathSymbol{N}{\mathalpha}{operators}{`N}
\DeclareMathSymbol{O}{\mathalpha}{operators}{`O}
\DeclareMathSymbol{P}{\mathalpha}{operators}{`P}
\DeclareMathSymbol{Q}{\mathalpha}{operators}{`Q}
\DeclareMathSymbol{R}{\mathalpha}{operators}{`R}
\DeclareMathSymbol{S}{\mathalpha}{operators}{`S}
\DeclareMathSymbol{T}{\mathalpha}{operators}{`T}
\DeclareMathSymbol{U}{\mathalpha}{operators}{`U}
\DeclareMathSymbol{V}{\mathalpha}{operators}{`V}
\DeclareMathSymbol{W}{\mathalpha}{operators}{`W}
\DeclareMathSymbol{X}{\mathalpha}{operators}{`X}
\DeclareMathSymbol{Y}{\mathalpha}{operators}{`Y}
\DeclareMathSymbol{Z}{\mathalpha}{operators}{`Z}